\documentclass[conference]{IEEEtran}
\usepackage{cite}

\ifCLASSINFOpdf
\usepackage[pdftex]{graphicx}
\DeclareGraphicsExtensions{.pdf,.jpeg,.png}
\else
\fi
\usepackage[cmex10]{amsmath}
\usepackage[font=footnotesize]{subfig}

\usepackage{fixltx2e}
\usepackage{url}
\usepackage{caption}
\usepackage[ruled,vlined, linesnumbered, algo2e]{algorithm2e}
\usepackage{algpseudocode}
\usepackage{float}
\usepackage{amsthm}
\usepackage{booktabs}
\usepackage{multirow}
\usepackage[first=0,last=9]{lcg}

\usepackage{color, colortbl}
\usepackage{threeparttable}
\usepackage{algpseudocode}
\usepackage[autostyle]{csquotes}

\usepackage{supertabular}
\usepackage[hyperfootnotes=false]{hyperref}
\usepackage[hyperfootnotes=false]{hyperref}
\definecolor{mydarkblue}{rgb}{0,0.08,0.45}
\hypersetup{
	colorlinks=true,
	citecolor=mydarkblue,
	linkcolor=Orange,
	urlcolor=Blue}
\definecolor{LightCyan}{rgb}{0.88,1,1}
\definecolor{Gray2}{gray}{0.95}

\usepackage[usenames,dvipsnames]{xcolor}
\usepackage{pifont}% http://ctan.org/pkg/pifont

\usepackage{subfig}

\newcommand{\argmini}[1]{\underset{#1}{\operatorname{argmin}}\;}
\newcommand{\emp}[1]{\textbf{\color{blue}#1}\;}
\newtheorem{theoremx}{Theorem}

\begin{document}

%\mainmatter  % start of an individual contribution

% first the title is needed
\title{Fast Parallel Randomized Algorithm for Nonnegative Matrix Factorization with KL Divergence for Large Sparse Datasets}

% author names and affiliations
% use a multiple column layout for up to three different
% affiliations
\author{\IEEEauthorblockN{\ }
	\IEEEauthorblockA{
		%School of Knowledge Science\\
		%Japan Advanced Institute of Science and Technology, Japan
	}}
	
	% conference papers do not typically use \thanks and this command
	% is locked out in conference mode. If really needed, such as for
	% the acknowledgment of grants, issue a \IEEEoverridecommandlockouts
	% after \documentclass
	
	% for over three affiliations, or if they all won't fit within the width
	% of the page, use this alternative format:
	% 
	\author{\IEEEauthorblockN{
			Duy Khuong Nguyen,
			Tu Bao Ho
		}
	}

%	\author{\IEEEauthorblockN{
%			Duy Khuong Nguyen\IEEEauthorrefmark{1}\IEEEauthorrefmark{2},
%			Tu Bao Ho \IEEEauthorrefmark{1},
%			Author3\IEEEauthorrefmark{3},
%		}
%		
%		\IEEEauthorblockA{
%			\IEEEauthorrefmark{1}
%			Japan Advanced Institute of Science and Technology
%		}
%		\IEEEauthorblockA{
%			\IEEEauthorrefmark{2}
%			University of Engineering and Technology, VNU, Hanoi
%		}
%		\IEEEauthorblockA{
%			\IEEEauthorrefmark{3}
%			University of ....
%		}
%	}

\maketitle

\begin{abstract}
Nonnegative Matrix Factorization (NMF) with Kullback-Leibler Divergence~(NMF-KL) is one of the most significant NMF problems and equivalent to Probabilistic Latent Semantic Indexing~(PLSI), which has been successfully applied in many applications. For sparse count data, a Poisson distribution and KL divergence provide sparse models and sparse representation, which describe the random variation better than a normal distribution and Frobenius norm. Specially, sparse models provide more concise understanding of the appearance of attributes over latent components, while sparse representation provides concise interpretability of the contribution of  latent components over instances. However, minimizing NMF with KL divergence is much more difficult than minimizing NMF with Frobenius norm; and sparse models, sparse representation and fast algorithms for large sparse datasets are still challenges for NMF with KL divergence. In this paper, we propose a fast parallel randomized coordinate descent algorithm having fast convergence for large sparse datasets to archive sparse models and sparse representation. The proposed algorithm's experimental results overperform the current studies' ones in this problem. 

\keywords{\ Nonnegative Matrix Factorization, Kullback-Leibler Divergence, Sparse Models, and Sparse Representation}
\end{abstract}

\section{Introduction}
The development of technology has been generating big datasets of count sparse data such as documents and social network data, which requires fast effective algorithms to manage this huge amount of information.  One of these tools is nonnegative matrix factorization~(NMF) with KL divergence, which is proved to be equivalent with Latent Semantic Indexing~(PLSI)~\cite{Ding2006a}. 

NMF is a powerful linear technique to reduce dimension and to extract latent topics, which can be readily interpreted to explain phenomenon in science~\cite{Lee1999,Gillis2014The,Sotiras2015finding}. NMF makes post-processing algorithms such as classification and information retrieval faster and more effective.  In addition, latent factors extracted by NMF can be more concisely interpreted than other linear methods such as PCA and ICA~\cite{Sotiras2015finding}. In addition, NMF is flexible with numerous divergences to adapt a large number of real applications~\cite{Wang2013,Zhang2011a}. 

For sparse count data, NMF with KL divergence and a Poisson distribution may provide sparse models and sparse representation describing better the random variation rather than NMF with Frobenius norm and a normal distribution~\cite{Sra2008non}. For example, the appearance of words over latent topics and of topics over documents should be sparse. However, achieving sparse models and sparse representation is still a major challenge because minimizing NMF with KL divergence is much more difficult than NMF with Frobenius norm~\cite{Kuang2015}.

In the NMF-KL problem, a given nonnegative data matrix $V \in \mathcal{R}_+^{n\times m}$ must be factorized into a product of two nonnegative matrices, namely a latent component matrix $W \in \mathcal{R}_+^{r\times n}$ and a representation matrix $F \in \mathcal{R}_+^{r\times m}$, where $n$ is the dimension of a data instance, $m$ is the number of data instances, and $r$ is the number of latent components or latent factors. The quality of this factorization is controlled by the objective function with KL divergence as follows:
{\small
\begin{equation}
D(V\|W^TF) = \sum\limits_{i=1}^{n} \sum\limits_{j=1}^{m} (V_{ij}\log{\frac{V_{ij}}{(W^TF)_{ij}}} - V_{ij} + (W^TF)_{ij})
\end{equation}
}

In the general form of $L_1$ and $L_2$ regularization variants, the objective function is written as follows:

{\small
\begin{equation}
D(V\|W^TF) + \frac{\alpha_1}{2} \|W\|^2_2 + \frac{\alpha_2}{2} \|F\|^2_2 + \beta_1 \|W\|_1 + \beta_2 \|F\|_1
\end{equation}
}

NMF with KL divergence has been widely applied in many applications for dense datasets. For example, spatially localized, parts-based subspace representation of visual patterns is learned by local non-negative matrix factorization with a localization constraint (LNMF)~\cite{Li2001}. In another study, multiple hidden sound objects from a single channel auditory scene in the magnitude spectrum domain can be extracted by NMF with KL divergence~\cite{Smaragdis2004}. In addition, two speakers in a single channel recording can be separated by NMF with KL divergence and $L_1$ regularization on $F$~\cite{Schmidt2007speech}.

However, the existing algorithms for NMF with KL divergence~(NMF-KL) are extremely time-consuming for large count sparse datasets. Originally, Lee and Seung, 2001~\cite{Lee2001} proposed the first multiple update iterative algorithm based on gradient methods for NMF-KL. Nevertheless, this technique is simple and ineffective because it requires a large number of iterations, and it ignores negative effects of nonnegative constraints. In addition, gradient methods have slow convergence for complicated logarithmic functions like KL divergence. Subsequently, Cho-Jui \& Inderjit, 2011~\cite{Hsieh2011} proposed a cycle coordinate descent algorithm having low complexity of one variable update. However, this method contains several limitations: first, it computes and stores the dense product matrix $W^TF$ although $V$ is sparse; second, the update of $W^TF$ for the change of each cell in $W$ and $F$ is considerably complicated, which leads to practical difficulties in parallel and distributed computation; finally, the sparsity of data is not considered, while large datasets are often highly sparse.

In comparison with NMF with Frobenius norm, NMF with KL divergence is much more complicated because updating variables will influence derivatives of other variables; this computation is extremely expensive. Hence, it is difficult to employ fast algorithms having multiple variable updates, which limits the number of effective methods.

In this paper, we propose a new advanced version of coordinate descent methods with significant modifications for large sparse datasets. Regarding the contributions of this paper, we:
\begin{itemize}
	\item Propose a fast sparse randomized coordinate descent algorithm using limited internal memory for nonnegative matrix factorization for huge sparse datasets, the full matrix of which can not stored in the internal memory. In this optimization algorithm, variables are randomly selected with uniform sampling to balance the order priority of variables. Moreover, the proposed algorithm effectively utilizes the sparsity of data, models and representation matrices to improve its performance. Hence, the proposed algorithm can be considered an advanced version of cycle coordinate descent for large sparse datasets proposed in~\cite{Hsieh2011}. 
	\item Design parallel algorithms for combinational variants of $L_1$ and $L_2$ regularizations.
	\item Indicate that the proposed algorithm using limited memory can fast attain sparse models, sparse representation, and fast convergence by evaluational experiments, which is a significant milestone in this research problem for large sparse datasets.
\end{itemize}

The rest of the paper is organized as follows. Section~\ref{sec:algorithm} presents the proposed algorithms. The theoretical analysis of convergence and complexity is discussed in Section~\ref{sec:analysis}. Section~\ref{sec:evaluation} shows the experimental results, and Section~\ref{sec:conclusion} summarizes the main contributions of this pape and discussion.

\section{Proposed Algorithm}\label{sec:algorithm}

In this section, we propose a fast sparse randomized coordinate descent parallel algorithm for nonnegative sparse matrix factorization on Kullback-Leibler divergence. 
We employ a multiple iterative update algorithm like EM algorithm, see Algorithm~\ref{algo:IMU}, because $D(V\|W^TF)$ is a non-convex function although it is a convex function when fixing one of two matrices $W$ and $F$. This algorithm contain a \textit{while} loop containing two main steps: the first one is to optimize the objective function by $F$ when fixing $W$; and the another one is to optimize the objective function by $W$ when fixing $F$. Furthermore, 
in this algorithm, we need to minimize Function~\ref{eq:obj}, the decomposed elements of which can be independently optimized in Algorithm~\ref{algo:IMU}:

{\small
\begin{equation}\label{eq:obj}
D(V\|W^TF) = \sum\limits_{i=1}^{m} D(V_i\|W^TF_i) = \sum\limits_{j=1}^{n} D(V^T_j\|F^TW_j) 
\end{equation}
}

Specially, a number of optimization problems $D(V_i\|W^TF_i)$  or $D(V^T_j\|F^TW_j)$ in Algorithm~\ref{algo:IMU} with the form $D(v\|Ax)$ can be independently and simultaneously solved by Algorithm~\ref{algo:sparse}. In this paper, we concern combinational variants of NMF KL divergence with $L_1$ and $L_2$ regularizations in the general formula, Function~\ref{eq:kl}:

\begin{algorithm2e}[h]
	\caption{Iterative multiplicative update}
	\label{algo:IMU} 
	\KwIn{$V \in \mathcal{R}_+^{n\times m}, r$, and $\alpha_1, \alpha_2, \beta_1, \beta_2 \geq 0$}
	\KwOut{$W \in \mathcal{R}_+^{n \times r}$, $F \in \mathcal{R}_+^{r \times m}$.}
	\Begin{
		Randomize $W \in \mathcal{R}_+^{r \times n}$\;
		Randomize $F \in \mathcal{R}_+^{r \times m}$\;
		\While{convergence condition is not satisfied}{
			$ids =$ a randomized ordered set of values $\{1, 2, ..., r\}$\;
			$sumW = W \text{\textbf{1}}^n$\;
			/*Optimizing the objective function by $F$ when fixing $W$*/\;
			\For{j = 1 to m}{
				/*Call Algorithm~\ref{algo:sparse} in parallel*/\;
				$F_j^{k+1} = $ Algorithm~\ref{algo:sparse} ($V_j$, $W^T$, $sumW$, $F_j^{k}$, $\alpha_2$, $\beta_2$, $ids$)
			}
			$sumF = F \text{\textbf{1}}^m$\;
			/*Optimizing the objective function by $W$ when fixing $F$*/\;
			\For{i = 1 to n}{
				/*Call Algorithm~\ref{algo:sparse} in parallel*/\;
				$W_i^{k+1} = $  Algorithm~\ref{algo:sparse} ($V^T_i$ $F^T$, $sumW$, $W_i^{k}$, $\alpha_2$, $\beta_2$, $ids$)\;
			}
		}
		\Return{$(W^{k+1})^T$, $F^{k+1}$}\;
	}
\end{algorithm2e}

{\small
\begin{equation} \label{eq:kl}
f(x) = D(v\|Ax) =  \sum\limits_{i=1}^{n} (v_i\log{\frac{v_i}{[Ax]_i}}-v_{i}+[Ax]_i) + \frac{\alpha}{2} \|x\|^2_2 + \beta |x|_1
\end{equation}
}
\noindent where $v \in \mathcal{R}^n_+, A \in \mathcal{R}^{n \times r}_+, x \in \mathcal{R}^r_+$

Because the vector $v$ is given, minimizing Function~\ref{eq:kl} is equivalent to minimizing Function~\ref{eq:kloptimize}:

{\small
\begin{equation} \label{eq:kloptimize}
f(x) = D(v\|Ax) =  \sum\limits_{i=1}^{n} (-v_i\log{[Ax]_i}+[Ax]_i) + \frac{\alpha}{2} \|x\|^2_2 + \beta |x|_1
\end{equation}
}

From Equation~\ref{eq:kloptimize}, the first and second derivative of the variable $x_k$ are computed by Formula~\ref{eq:Partial}:

{\small
\begin{equation} \label{eq:Partial}
\Rightarrow \left\{  \begin{array}{ll}
\nabla f_k  & = - \sum\limits_{i=1}^{n} v_i\frac{A_{ik}}{[Ax]_i} + \sum\limits_{i=1}^{n} A_{ik} + \alpha x_k + \beta \\
\nabla^2 f_{kk}  & = \sum\limits_{i=1}^{n} v_i (\frac{A_{ik}}{[Ax]_i})^2 + \alpha \\
\end{array}\right.
\end{equation}
}

Based on Formula~\ref{eq:Partial}, we have several significant remarks:
\begin{itemize}
	\item One update of $x_k$ changes all elements of $Ax$, which are under the denominators of fractions. Hence, it is difficult to employ fast algorithms having simultaneous updates of multiple variables because it will require heavy computation. Hence, we employ coordinate descent methods to reduce the complexity of each update, and to avoid negative effects of nonnegative constraints.
	\item One update of $x_k$ has complexity of maintaining $\nabla f_k$ and $\nabla^2 f_{kk}$ as $\mathcal{O}(k+\text{nnz}(v)+\text{nnz}(v)) = k + \text{nnz}(v)$ if $\sum_{i=1}^{n} A_{ik}$ is computed in advance. Specially, it employs $\mathcal{O}(k+\text{nnz}(v))$  of multiple and addition operators, and exactly $\mathcal{O}(\text{nnz}(v))$ of divide operators; where $\text{nnz}(v)$ is the number of non-zero elements in the vector $v$. Hence, for sparse datasets, the number of operators can be negligible.
	\item The used internal memory of Algorithm~\ref{algo:IMU} and Algorithm~\ref{algo:sparse} is $\mathcal{O}(\text{nnz}(V)+\text{size}(W)+ \text{size}(F))=\text{nnz}(V)+(n+m)r$, where $\text{nnz}(V)$ is the number of non-zero elements in the given matrix $V$, which is much smaller than $\mathcal{O}(mn+(n+m)r)$ for the existing algorithms~\cite{Hsieh2011,Lee2001}.
\end{itemize}

Hence, Algorithm~\ref{algo:sparse} employs a coordinate descent algorithm based on projected Newton methods with quadratic approximation in Algorithm~\ref{algo:sparse} is to optimize Function~\ref{eq:kloptimize}. Specially, because Function~\ref{eq:kloptimize} is convex, a coordinate descent  algorithm based on projected Newton method~\cite{Luenberger2008linear,Luenberger2008linear} with quadratic approximation is employed to iteratively update  with the nonnegative lower bound as follows:
\begin{equation*}
x_k = \max(0, x_k - \frac{\nabla f_k}{\nabla^2 f_{kk}})
\end{equation*}

Considering the limited internal memory and the sparsity of $x$, we maintain $W^TF_j$ via computing $Ax$ instead of storing the dense matrix $W^TF$ for the following reasons:
\begin{itemize}
	\item The internal memory requirement will significantly decrease, so the proposed algorithm can stably run on limited internal memory machines.
	\item The complexity of computing $Ax$ is always smaller than the complexity of computing and maintaining  $\nabla f_k$ and $\nabla^2 f_{kk}$, so it does not cause the computation more complicated.
	\item The updating $Ax = Ax + \triangle x A_k$ as the adding with a scale of two vectors utilizes the speed of CPU cache because of accessing consecutive memory cells.
	\item The recomputing helps remove the complexity of maintaining the dense product  matrix $W^TF$ as in~\cite{Hsieh2011,Lee2001}, which is certainly considerable because this maintenance accesses memory cells far together and does not utilize CPU cache.
\end{itemize}

\begin{algorithm2e}[h]
	\caption{Randomized coordinate descent algorithm for sparse datasets}
	\label{algo:sparse} 
	\KwIn{$v \in \mathcal{R}^n$, $A \in \mathcal{R}^{n \times k}$, $sumA$, $x \in \mathcal{R}^k$, $\alpha \geq 0$, $\beta \geq 0$, and variable order $ids$}
	\KwOut{$x$ is updated by $x \approx \argmini{x \succeq 0} \sum\limits_{i=1}^{n} -v_i\log([Ax]_i+\epsilon) + [Ax]_i + \frac{\alpha}{2}\|x\|_2^2 + \beta\|x\|_1$.}
	\Begin{
		Compute $Ax \in \mathcal{R}^n$\label{line:axAx}\; 
		\For{$k$ in order ids}{ 
			Compute $\nabla f_k$ and $\nabla^2 f_{kk}$\ based on $\alpha, \beta, v, x, Ax, sumA$ and sparsity of $v$ based on Formula~\ref{eq:Partial}\; \label{line:dfddf}
			\While{$(\nabla f_k < -\epsilon)$ or ($|\nabla f_k| > \epsilon$  and $x_k > \epsilon$)}{
				$\triangle x = max(0, x_k - \frac{\nabla f_k}{\nabla^2 f_{kk}}) - x_k$\;
				Update $Ax$ via $Ax = Ax + \triangle x A_k$\;
				$xs = x_k$\;
				$x_k = x_k + \triangle x$\;
				\If{$(\triangle x < \epsilon_x xs)$}{break;}
				Update $\nabla f_k$ and $\nabla^2 f_{kk}$ based on $\alpha, \beta, v, x, sumA$ and sparsity of $v$ based on Formula~\ref{eq:Partial}\; \label{line:redfddf}
			}
		}
		\Return{$x$}\;
	}
\end{algorithm2e}

In summary, in comparison with the original coordinate algorithm~\cite{Hsieh2011} for NMK-KL, the proposed algorithm involve significant improvements as follows: 
\begin{itemize}
	\item Randomize the order of variables to optimize the objective function in Algorithm~\ref{algo:sparse}. Hence, the proposed algorithm can balance the order priority of variables,
	\item Remove duplicated computation of maintaining derivatives $\nabla f_k$ and $\nabla^2 f_{kk}$  by computing common elements $sumW = W \text{\textbf{1}}^n$ and $sumF = F \text{\textbf{1}}^m$ in advance, which led to that the complexity of computing $\nabla f_k$ and $\nabla^2 f_{kk}$ only depends on the sparsity of data,
	\item Effectively utilize the sparsity of $W$ and $F$ to reduce the running time of computing $Ax$, 
	\item Effectively utilize CPU cache to improve the performance of maintaining $Ax = Ax + \triangle x A_k$,
	\item Recompute $Ax$ but remove the maintenance of the dense matrix product $W^TF$. Hence, the proposed algorithm stably run on the limited internal memory systems with the required memory size $\mathcal{O}(\text{nnz}(V)+\text{size}(W)+ \text{size}(F)) = \text{nnz}(V)+(m+n)r$, which is much smaller than $\mathcal{O}(mn+(n+m)r)$ for the existing algorithms~\cite{Hsieh2011,Lee2001}.
\end{itemize}

\section{Theoretical Analysis}\label{sec:analysis}

In this section, we analyze the convergence and complexity of Algorithm~\ref{algo:IMU} and Algorithm~\ref{algo:sparse}. 

In comparison with the previous algorithm of Hsieh \& Dhillon, 2011~\cite{Hsieh2011}, the proposed algorithm has significant  modifications for large sparse datasets  by means of adding the order randomization of indexes and utilizing the sparsity of data $V$, model $W$,  and representation $F$. These modifications does not affect on the convergence guarantee of algorithm. Hence, based on Theorem 1 in Hsieh \& Dhillon, 2011~\cite{Hsieh2011}, Algorithm~\ref{algo:sparse} converges to the global minimum of $f(x)$. Furthermore, based on Theorem 3 in Hsieh \& Dhillon, 2011~\cite{Hsieh2011}, Algorithm~\ref{algo:IMU} using Algorithm~\ref{algo:sparse} will converge to a stationary point. In practice, we set $\epsilon_x = 0.1$ in Algorithm~\ref{algo:sparse}, which is more precise than $\epsilon_x = 0.5$  in Hsieh \& Dhillon, 2011~\cite{Hsieh2011}.

Concerning the complexity of Algorithm~\ref{algo:sparse}, based on the remarks in Section~\ref{sec:algorithm}, we have Theorem~\ref{theo:complexity}. Furthermore, because KL divergence is a convex function over one variable and the nonnegative domain, and project Newton methods with quadratic approximation for convex functions have superlinear rate of convergence~\cite{Bertsekas1982projected,Luenberger2008linear}, the average number of iterations $\bar{t}$ is small.

\begin{theoremx}\label{theo:complexity}
	The complexity of Algorithm~\ref{algo:sparse} is $\mathcal{O}(n\text{nnz}(r)+\bar{t}r(r+n+\text{nnz}(n)))$, where $\text{nnz}(r)$ is the number of non-zero elements in $x$, $\text{nnz}(n)$ is the number of non-zero elements in $v$, and $\bar{t}$ is the average number of iterations. Then, the complexity of a \textit{while} iteration in Algorithm~\ref{algo:IMU} is $\mathcal{O}(\bar{t}(mnr + (m+n)r^2))$
\end{theoremx}

\begin{proof}
Consider the major computation in Algorithm~\ref{algo:sparse}, based on Formula~\ref{eq:Partial}, we have:
\begin{itemize}
	\item The complexity of  computing $Ax$ in Line~\ref{line:axAx} is $\mathcal{O}(n\text{nnz}(r))$,
	\item The complexity of  computing $\nabla f_k$ and $\nabla^2 f_{kk}$ in Line~\ref{line:dfddf} is $\mathcal{O}(r+\text{nnz}(n))$ because $Ax$ and $sumA$ are computed in advance,
	\item The complexity of  updating $\nabla f_k$ and $\nabla^2 f_{kk}$ in Line~\ref{line:redfddf} is $\mathcal{O}(r+n+\text{nnz}(n))$ because only one dimension of vector $x$ is changed.
\end{itemize}

Hence, 	the complexity of Algorithm~\ref{algo:sparse} is $\mathcal{O}(n\text{nnz}(r)+\bar{t}r(r+n+\text{nnz}(n)))$.

In addition, the complexity of computing $sumW$ and $sumF$ is  $\mathcal{O}((m+n)r)$. Hence, the complexity of a \textit{while} iteration in Algorithm~\ref{algo:IMU} is $\mathcal{O}((m+n)r+mn\text{nnz}(r)+\bar{t}mr(r+n+\text{nnz}(n))) \approx (m+n)r+\bar{t}(mnr + (m+n)r^2) \approx \bar{t}(mnr + (m+n)r^2)$
Therefore, we have Theorem~\ref{theo:complexity}
\end{proof}

For large sparse datasets, $m, n \gg r \Rightarrow \mathcal{O}(\bar{t}(mnr + (m+n)r^2)) \approx \bar{t}(mnr)$. This complexity is raised by the operators $Ax = Ax + \triangle x A_k$ in Algorithm~\ref{algo:sparse}. To reduce the running time of these operators, $A_k$ must be stored in an array to utilize CPU cache memory by accessing continuous memory cells of $Ax$ and $A_k$.

\section{Experimental Evaluation}\label{sec:evaluation}

In this section, we investigate the effectiveness of the proposed algorithm via convergence and sparsity. Specially, we compare the proposed algorithm Sparse Randomized Coordinate Descent (\textbf{SRCD}) with state-of-the-art algorithms as follows:
\begin{itemize}
	\item Multiplicative Update (\textbf{MU})~\cite{Lee2001}: This algorithm is the original method for NMF with KL divergence.
	\item Cycle Coordinate Descent (\textbf{CCD})~\cite{Hsieh2011}: This algorithm has the current fastest convergence because it has very low complexity of each update for one variable.
\end{itemize}

\textbf{Datasets:} To investigate the effectiveness of the algorithms compared, the 4 sparse datasets used are shown in Table~\ref{tab:datasets}. The dataset Digit is downloaded from~\footnote{\url{http://yann.lecun.com/exdb/mnist/}}, and the other tf-idf datasets Reuters21578, TDT2, and RCV1\_4Class are downloaded from~\footnote{\url{http://www.cad.zju.edu.cn/home/dengcai/Data/TextData.html}}.

\begin{table}
	\caption{Summary of datasets}
	\label{tab:datasets}
	\centering
	{\begin{tabular}{lcccc} 
			\hline\noalign{\smallskip}
			Dataset ($V$) & $n$ & $m$ & $\text{nnz}(V)$ & Sparsity (\%) \\
			\noalign{\smallskip}\hline\noalign{\smallskip}
			Digits 			 & $784$ & $60,000$  & $8,994,156$ & $80.8798$\\
			Reuters21578 & $8,293$ & $18,933$  & $389,455$ & $99.7520$\\
			TDT2 	& $9,394$ & $36,771$  & $1,224,135$ & $99.6456$\\
			RCV1\_4Class & $9,625$ & $29,992$ & $730,879$ & $99.7468$ \\
			\noalign{\smallskip}\hline
		\end{tabular}}
	\end{table}

	\textbf{Environment settings}: We develop the proposed algorithm SRCD in Matlab with embedded code C++ to compare them with other algorithms. We set system parameters to use only 1 CPU for Matlab and the IO time is excluded in the machine Mac Pro 8-Core Intel Xeon E5 3 GHz 32GB. In addition, the  initial matrices $W^0$ and $F^0$ are set to the same values. The source code will be published on our homepage~\footnote{\url{http://khuongnd.appspot.com/}}.
	
\subsection{Convergence}

	In this section, we investigate the convergence of the objective value $D(V||W^TF)$ versus running time by running all the compared algorithms on the four datasets with two numbers of latent factors $r = 10$ and $r = 20$. The experimental results are depicted in Figure~\ref{fig:nmfkl_convergence10} and Figure~\ref{fig:nmfkl_convergence20}. From these figures, we realize two significant observations as follows:
	\begin{itemize}
		\item The proposed algorithm (SRCD) has much faster convergence than the algorithms CCD and MU,
		\item The sparser the datasets are, the greater the distinction between the convergence of the algorithm SRCD and the other algorithms CCD and MU is. Specially, for Digits with 81\% sparsity, the algorithm SRCD's convergence is lightly faster than the convergence of the algorithms CCD and MU. However, for three more sparse datasets Reuters21578, TDT2, and RCV1\_4Class with above 99\% sparsity, the distance between these convergence speeds is readily apparent.
	\end{itemize}
	
	\begin{figure}
		\centering{
			\hspace*{-10pt}
			\includegraphics[scale=0.68]{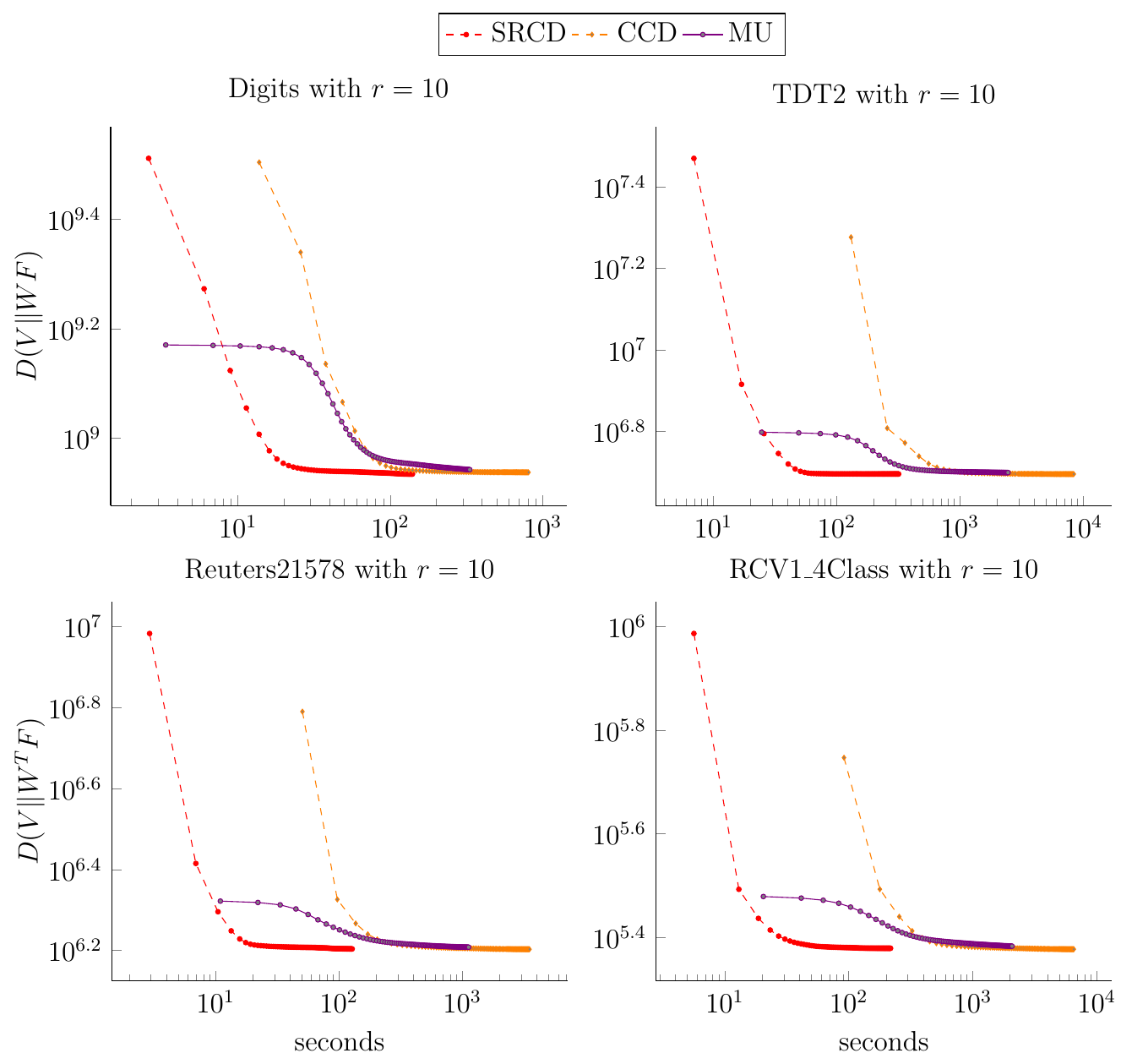}
			\caption{Objective value $D(V\|W^TF)$ versus running time with $r = 10$}
			\label{fig:nmfkl_convergence10}}
	\end{figure}
	
	\begin{figure}
		\centering{
			\hspace*{-10pt}
			\includegraphics[scale=0.68]{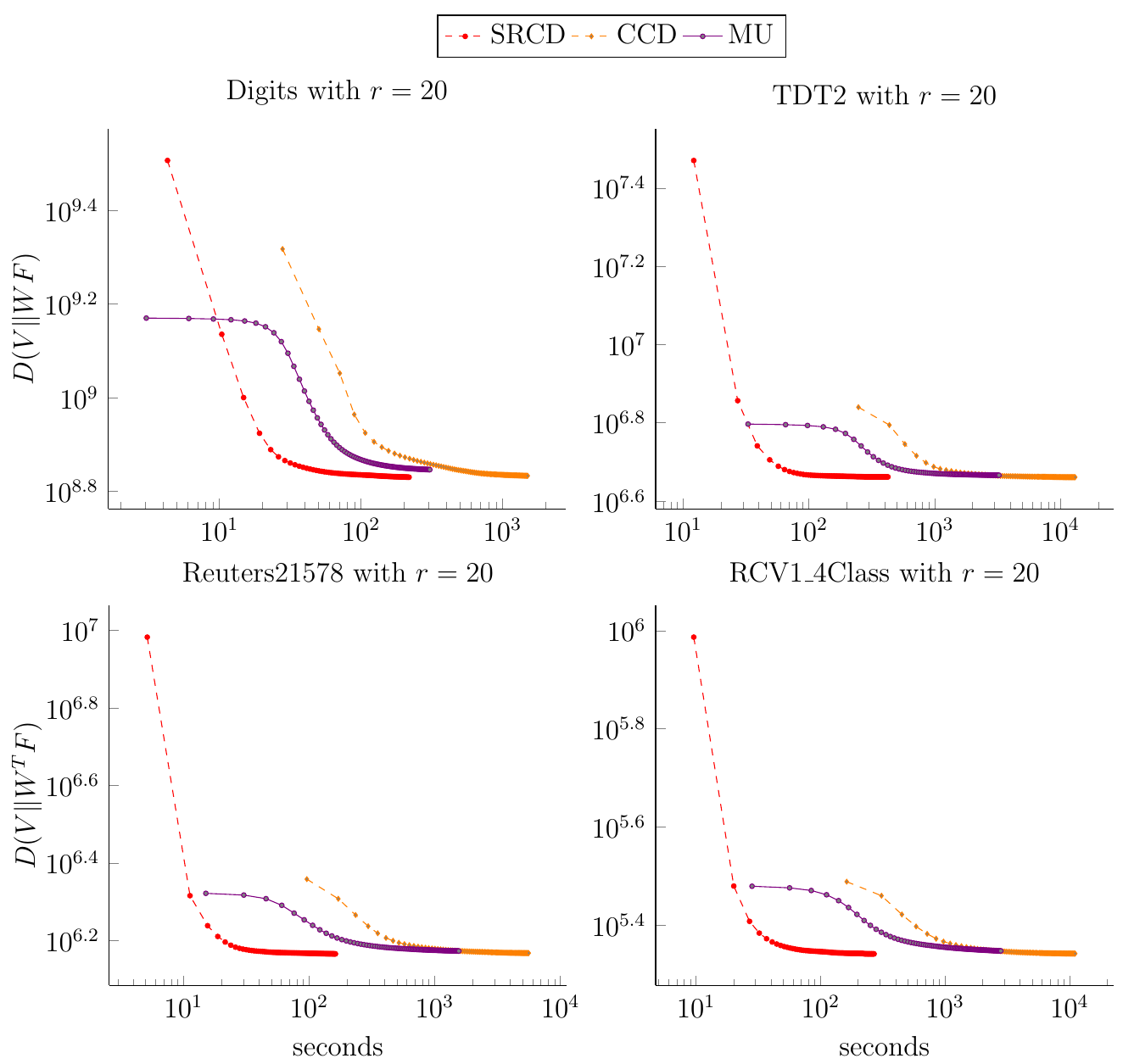}
			\caption{Objective value $D(V\|W^TF)$ versus running time with $r = 20$}
			\label{fig:nmfkl_convergence20}}
	\end{figure}
	
\subsection{Sparsity of factor matrices}

Concerning the sparsity of factor matrices $W$ and $F$, the algorithms CCD and MU does not utilize the sparsity of factor matrices. Hence, these algorithms add a small number into these factor matrices to obtain convenience in processing special numerical cases. Hence, the sparsity of factor matrices $W$ and $F$ for the algorithms CCD and MU both are 0\%. Although this processing may not affect other post-processing tasks such as classification and information retrieval, it will reduce the performance of these algorithms. The sparsity of $(W, F)$ of the proposed algorithm's results is showed in Table~\ref{tab:sparsity}. These results clearly indicate that the sparse model $W$ and the sparse representation $F$ are attained. The results also explain why the proposed algorithm runs very fast on the sparse datasets Reuters21578, TDT2 and RCV1\_4Class, when it can obtain highly sparse models and sparse representation in these highly sparse datasets.

\begin{table}
	\caption{Sparsity (\%) of $(W, F)$ for the algorithm SRCD's results}
	\label{tab:sparsity}
	\centering
	{\begin{tabular}{lcccc} 
			\hline\noalign{\smallskip}
			\   & Digits & Reuters21578 & TDT2 & RCV1\_4Class\\
			\noalign{\smallskip}\hline\noalign{\smallskip}
			$r = 10$  & (74.3, 49.2)& (75.6, 71.6) & (68.5, 71.3)& (81.2, 74.0) \\
			$r = 20$  & (87.8, 49.7)& (84.2, 80.4)& (78.6, 81.1)& (88.4, 83.0)\\
			\noalign{\smallskip}\hline
		\end{tabular}}
	\end{table}

\subsection{Used internal memory}

\begin{table}
	\caption{Used internal memory (GB) for $r = 10$}
	\label{tab:nmfkl_memory}
	\centering
	{\begin{tabular}{lccc} 
			\hline\noalign{\smallskip}
			Datasets &	MU &	CCD	& SRCD \\
			\noalign{\smallskip}\hline\noalign{\smallskip}
			Digits	& 1.89 & \emp{0.85} &	1.76\\
			Reuters21578 &	5.88 &	2.46 &	\emp{0.17}\\
			TDT2 & 11.51 & 5.29 & \emp{0.30}\\
			RCV1\_4Class &	9.73	& 4.43 & \emp{0.23}\\
			\noalign{\smallskip}\hline
		\end{tabular}}
	\end{table}
	
Table~\ref{tab:nmfkl_memory} shows the internal memory used by algorithms. From the table, we have two significant observations:
	\begin{itemize}
		\item For the dense dataset Digits, the proposed algorithm SRCD uses more internal memory than the the algorithm CCD because a considerable amount of memory is used for the indexing of matrices. 
		\item For the sparse datasets Reuters21578, TDT2, and RCV1\_4Class, the internal memory for SRCD is remarkably smaller than the internal one for MU and CCD. These results indicate that we can conduct the proposed algorithm for huge sparse datasets with a limited internal memory machine is stable. 
	\end{itemize}

\subsection{Running on large datasets}

This section investigates running the proposed algorithm on large datasets with different settings. Figure~\ref{fig:runningtime_rs} shows the running time of Algorithm SRCD for 100 iterations with different number of latent component using 1 thread. Clearly, the running time linearly increases, which fits the theoretical analyses about fast convergence and linear complexity for large sparse datasets in Section~\ref{sec:analysis}. Furthermore, concerning the parallel algorithm, the running time of Algorithm SRCD for 100 iterations significantly decreases when the number of used threads increases in Figure~\ref{fig:parallel_NMFKL}. In addition, the running time is acceptable for large applications. Hence, these results indicate that the proposed algorithm SRCD is feasible for large scale applications.

	\begin{figure}
		\centering
		\includegraphics[scale=0.6]{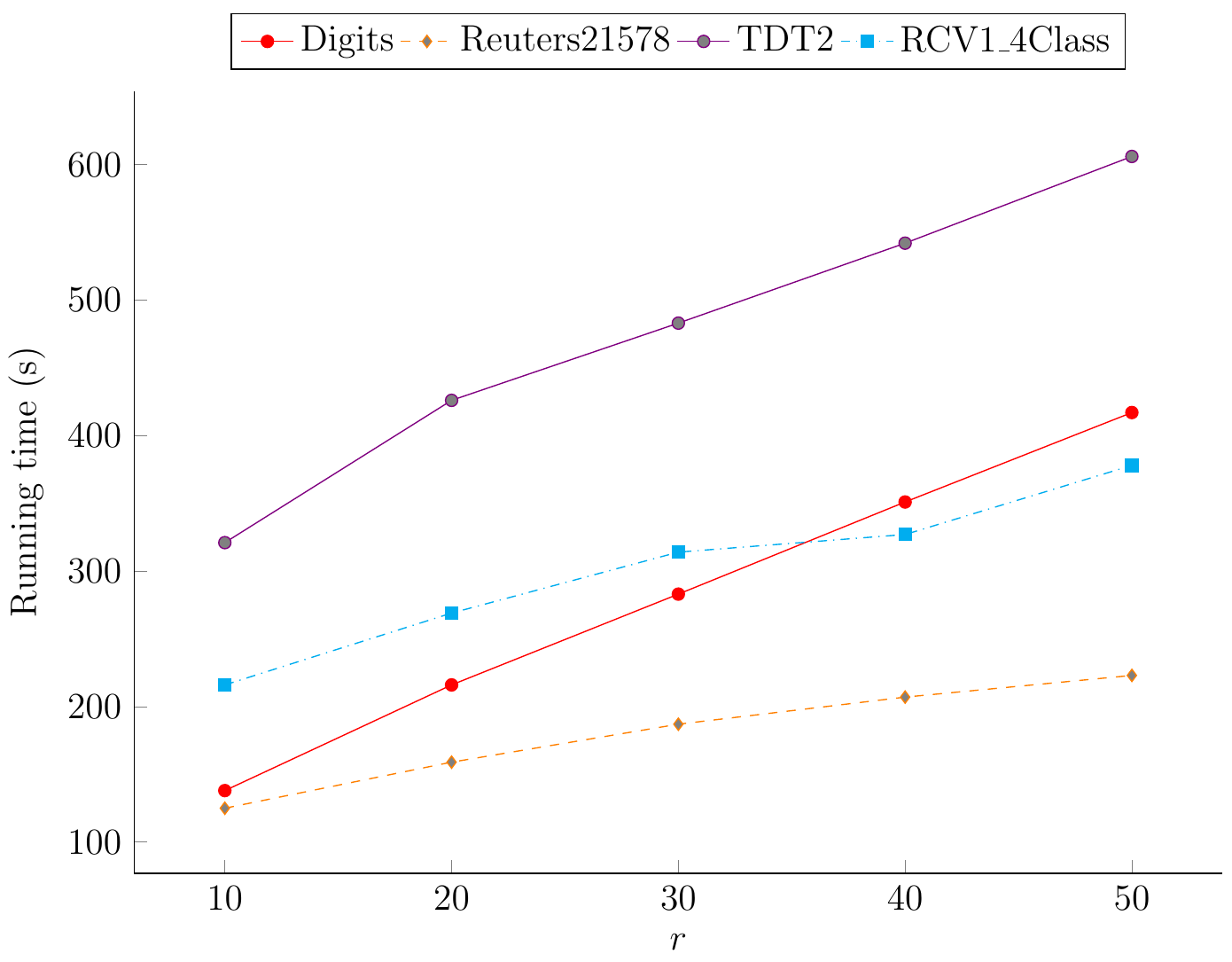}
		\caption{Running time of 100 iterations with different number of latent component using 1 thread}
		\label{fig:runningtime_rs}
	\end{figure}

	\begin{figure}
		\centering
		\includegraphics[scale=0.6]{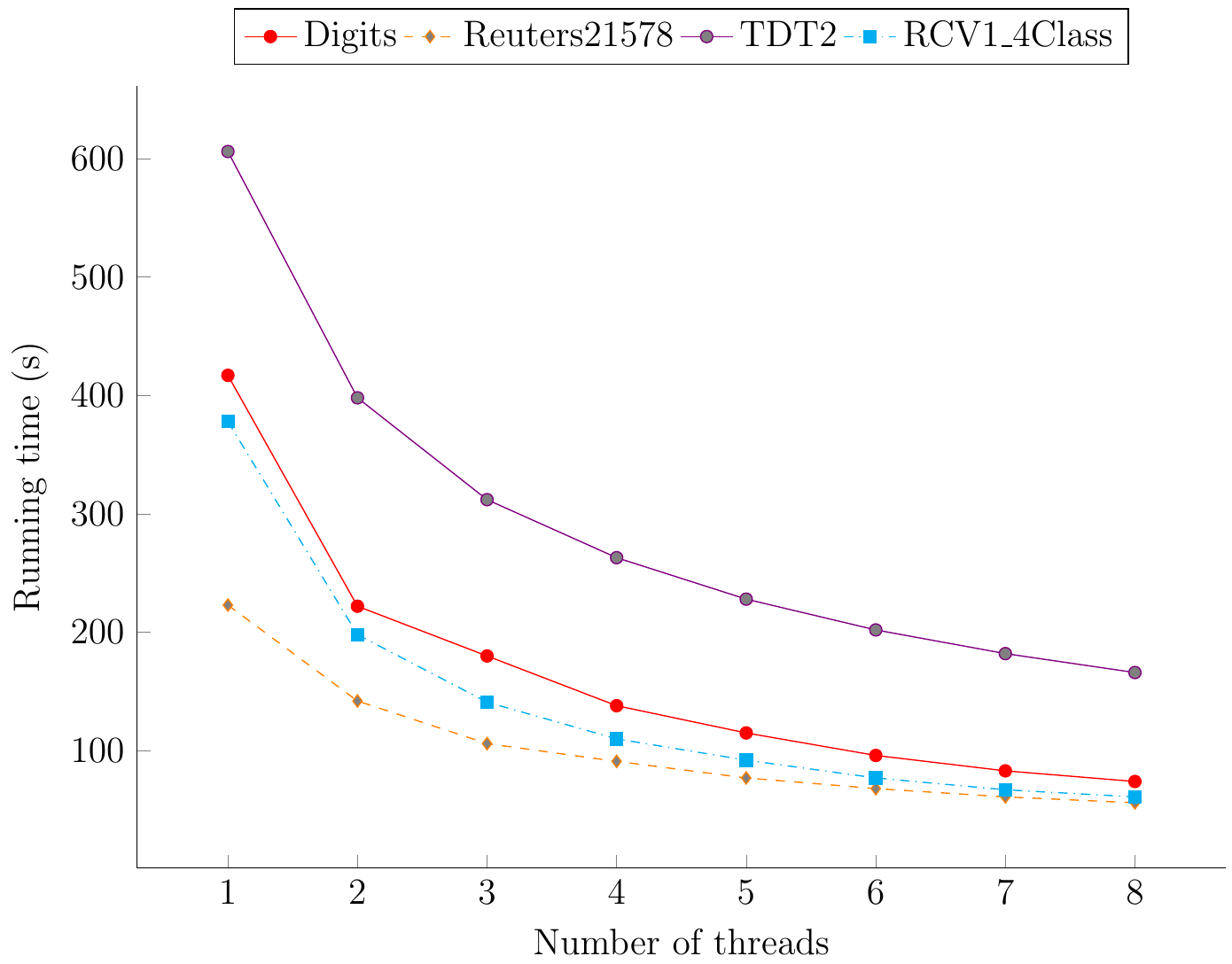}
		\caption{Running time of 100 iterations with $r = 50$ and using different number of threads}
		\label{fig:parallel_NMFKL}
	\end{figure}

\section{Conclusion and Discussion}\label{sec:conclusion}

In this paper, we propose a fast parallel randomized coordinate descent algorithm for NMF with KL divergence for large sparse datasets. The proposed algorithm attains fast convergence by means of removing duplicated computation, exploiting sparse properties of data, model and representation matrices, and utilizing the fast accessing speed of CPU cache. In addition, our method can stably run systems within limited internal memory by reducing internal memory requirements.  Finally, the experimental results indicate that highly sparse models and sparse representation can be attained for large sparse datasets, a significant milestone in researching this problem. In future research, we will generalize this algorithm for nonnegative tensor factorization.

\section*{Acknowledgements}
This work is partially sponsored by Asian Office of Aerospace R\&D under agreement number FA2386-15-1-4006, by funded by Vietnam National University at Ho Chi Minh City under the grant number B2015-42-02m, and by 911 Scholarship from Vietnam Ministry of Education and Training.

\bibliographystyle{abbrv}      % mathematics and physical sciences
\bibliography{references}   % name your BibTeX data base

\end{document}